\documentclass[12pt]{amsart}

\usepackage{tgpagella}
\usepackage{euler}
\usepackage[T1]{fontenc}
\usepackage{amsmath, amssymb}
\usepackage[hidelinks]{hyperref}
\usepackage[english]{babel}
\usepackage{mathrsfs}
\usepackage{eucal}
\usepackage[all]{xy}
\usepackage{tikz}

\newtheorem{thm}{Theorem}
\newtheorem*{thm*}{Theorem}

\newtheorem{prop}[thm]{Proposition}
\newtheorem*{prop*}{Proposition}

\newtheorem*{cor*}{Corollary}

\theoremstyle{definition}

\newtheorem*{defn*}{Definition}

\newtheorem{remark}[thm]{Remark}

\newtheorem{question}[thm]{Question}

\newtheorem*{question*}{Question}
\newtheorem*{Pquestion*}{Popa's question}

\newtheorem*{conv*}{Convention}

\def\u{\mathsf 1}

\newcommand{\cstar}{$\mathrm{C}^*$}

\def \tp{\operatorname{tp}}

\makeatletter

\def\dotminussym#1#2{%
  \setbox0=\hbox{$\m@th#1-$}%
  \kern.5\wd0%
  \hbox to 0pt{\hss\hbox{$\m@th#1-$}\hss}%
  \raise.6\ht0\hbox to 0pt{\hss$\m@th#1.$\hss}%
  \kern.5\wd0}

\def \R{\mathcal R}
\def \u{\mathcal U}

\begin{document}


\title{Non-embeddable II$_1$ factors resembling the hyperfinite II$_1$ factor}
\author{Isaac Goldbring}
\thanks{The author was partially supported by NSF CAREER grant DMS-1349399.}

\address{Department of Mathematics\\University of California, Irvine, 340 Rowland Hall (Bldg.\# 400),
Irvine, CA 92697-3875}
\email{isaac@math.uci.edu}
\urladdr{http://www.math.uci.edu/~isaac}

\maketitle

\begin{abstract}
We consider various statements that characterize the hyperfinite II$_1$ factors amongst embeddable II$_1$ factors in the non-embeddable situation.  In particular, we show that ``generically'' a II$_1$ factor has the Jung property (which states that every embedding of itself into its ultrapower is unitarily conjugate to the diagonal embedding) if and only if it is self-tracially stable (which says that every such embedding has an approximate lifting).  We prove that the enforceable factor, should it exist, has these equivalent properties.  Our techniques are model-theoretic in nature.  We also show how these techniques can be used to give new proofs that the hyperfinite II$_1$ factor has the aforementioned properties.
\end{abstract}

\section{introduction}

In \cite{mvn}, Murray and von Neumann proved that there exists a unique (up to isomorphism) separable hyperfinite II$_1$ factor.  This unique factor, henceforth denoted by $\R$, plays a crucial role in the theory of finite von Neumann algebras.  By Connes' seminal work in \cite{connes}, we know that $\R$ is also the unique separable II$_1$ factor possessing any of the following properties:  injectivity, semidiscreteness, and amenability.

In this article, our focus will be on some statements that characterize $\R$ amongst the class of separable \textbf{embeddable} II$_1$ factors, where a separable tracial von Neumann algebra is embeddable if it embeds into some (equivalently, any) ultrapower $\R^\u$ of $\R$ with $\u$ a nonprincipal ultrafilter on $\mathbb N$.  For example, as proven by Jung in \cite{jung}, any embedding of $\R$ into $\R^\u$ is unitarily conjugate to the diagonal embedding.  In \cite{agk}, the authors say that a separable II$_1$ factor $M$ has the \textbf{Jung property} if and only if any embedding of $M$ into $M^\u$ is unitarily conjugate to the diagonal embedding.  In \cite{ak} (see also \cite[Theorem 3.1.3]{agk}), the authors show that $\R$ is the unique separable embeddable II$_1$ factor with the Jung property.

In \cite{a}, the author defines a separable tracial von Neumann algebra $M$ to be \textbf{self-tracially stable} if any embedding of $M$ into $M^\u$ has an ``approximate lifting.''  (See the next section for a precise definition.)  It is easy to see that any II$_1$ factor with the Jung property is self-tracially stable (see \cite[Proposition 3.3.14]{agk} for a proof).  It follows that $\R$ is self-tracially stable.  The fact that $\R$ is the unique separable embeddable self-tracially stable II$_1$ factor is the content of \cite[Theorem 2.4]{ak}.

Recall that the \textbf{Connes Embedding Problem} (CEP) asks whether or not every separable\footnote{Separability is not an issue here if one allows ultrafilters over larger index sets.} tracial von Neumann algebra is embeddable.  As announced in the recent landmark paper \cite{mip}, the Connes Embedding Problem has a negative answer.  It thus makes sense to ask whether or not there are separable non-embeddable II$_1$ factors that have the Jung property or are self-tracially stable.  (See \cite[Question 3.3.12]{agk} for an explicit mention of the former question.)  \footnote{We should mention that in \cite[Theorem 3.2.5]{agk} it was shown that $\R$ is the unique separable embeddable II$_1$ factor with the \textbf{generalized Jung property}, meaning that any two embeddings of itself into its ultrapower are conjugate by some (not necessarily inner) automorphism of the ultrapower; \cite[Theorem 3.3.1]{agk} shows that there are non-embeddable factors with this property.}

Our first main result is that ``generically'' these are the same question.  To explain this, recall that a tracial von Neumann algebra $M$ is \textbf{existentially closed} (or e.c. for short) if:  whenever $M$ is contained in the tracial von Neumann algebra $N$, then there is an embedding of $N$ into $M^\u$ that restricts to the diagonal embedding of $M$.\footnote{Again, this definition makes sense for not necessarily separable factors using ultrfilters on larger index sets.  Alternatively, one can give a purely syntactical, model-theoretic, definition which makes it clear that density character is irrelevant.}  The notion of e.c. tracial von Neumann algebras comes from model theory and has proven useful in applications of model theory to operator algebras.  Much is known about the class of e.c. tracial von Neumann algebras:  they must be McDuff II$_1$ factors with only approximately inner automorphisms (see \cite{ecfactor} for more on this class).  The generic separable tracial von Neumann algebra is e.c. in the sense that in a natural Polish topology on the space of separable tracial von Neumann algebras, the e.c. algebras form a comeager set.  The notion of e.c. factor can be relativized to the class of embeddable factors, in which case $\R$ is an e.c. embeddable factor.  \footnote{This follows immediately from the fact that $\R$ has the Jung property, but we will discuss another proof at the end of this paper.}

We show the following:

\begin{thm}
If $M$ is a separable e.c. factor, then any embedding of $M$ into itself is approximately inner. 
\end{thm}

From this theorem, it follows fairly quickly that if $M$ is separable, e.c., and self-tracially stable, then $M$ has the Jung property; see Proposition 6 below.

We next turn to two model theoretic characterizations of $\R$ amongst embeddable factors.  We call an e.c.\ (embeddable) factor \textbf{enforceable} if it embeds into all other e.c.\ (embeddable) factors.\footnote{This is not the original definition given in \cite{games}, but is equivalent by the results in Section 6 of that paper.}  Should the enforceable (embeddable) factor exist, it is automatically unique.  In \cite[Theorems 5.1 and 5.2]{games}, it is shown that $\R$ is the enforceable embeddable factor and that the CEP is equivalent to $\R$ being the enforceable factor.  Due to the negative solution of the CEP, we see that $\R$ is not the enforceable factor.  This does not, however, preclude the existence of the enforceable factor.  We view the problem of the existence of the enforceable factor to be one of the central problems in the model theory of II$_1$ factors, for if the enforceable factor exists, then it is a canonical object deserving of further study, whereas any proof that it does not exist yields a stronger refutation of CEP.

In this paper, we prove:

\begin{thm}
If the enforceable factor exists, then it has the Jung property.
\end{thm}

It is worth noting that by \cite[Theorem 2.14]{ssa}, if $M$ is a II$_1$ factor with the Jung property and $M$ is \textbf{elementarily equivalent} to $M\otimes M$, then $M\cong \R$.\footnote{Here, two II$_1$ factors are elementarily equivalent if they have the same first-order theory.  By the Keisler-Shelah Theorem, we can equivalently say that they have isomorphic ultrapowers.  The reference \cite{ssa} actually only deals with strongly self-absoring \cstar-algebras, but the proof there implies the result that we mention above.}  Consequently, if the enforceable factor $\mathcal{E}$ exists, then we have that $\mathcal{E}$ is not elementarily equivalent to $\mathcal{E}\otimes \mathcal{E}$.\footnote{By the work in \cite[Remark 5.8]{games}, the failure of CEP already implies that $\mathcal{E}$, should it exist, could not be isomorphic to $\mathcal{E}\otimes \mathcal{E}$.}

Our final result concerns the \textbf{finite forcing companion}.  A \textbf{finitely generic factor} is a particular kind of e.c. II$_1$ factor with the generalized Jung property.  (See \cite[Definition 5.3]{games} for a precise definition.  Alternatively, \cite[Propsition 3.10]{games} presents a more workable version of the notion.)  These factors always exist and any two of them are elementarily equivalent; the common first-order theory of the finitely generic factors is known as the finite forcing companion, denoted $T^f$.  In the embeddable situation, $\R$ is a finitely generic embeddable factor (see \cite[Corollary 3.14]{games}), whence the finite forcing companion is simply the complete theory of $\R$.  Since $\R$ embeds in any model of its theory (due to the axiomatizability of being McDuff), the fact that $\R$ has the generalized Jung property implies that it is the \textbf{prime model} of its theory.\footnote{In general, the prime model of a theory is a model which elementarily embeds into any other model of the theory and, if it exists, it is automatically unique.}  

It is currently unknown whether or not $T^f$ has a prime model.  However, if it does, then it is a non-embeddable factor with the Jung property:

\begin{thm}
If $T^f$ has a prime model $M$, then $M$ has the Jung property.
\end{thm}

In the final section, we revisit the embeddable situation and give model-theoretic proofs that $\R$ has the Jung property and is self-tracially stable that might be of independent interest.

In order to keep this note fairly short, we will freely use model-theoretic language when necessary.  The reader is advised to consult \cite[Section 2]{agk} for a more thorough introduction.

We would like to thank Scott Atkinson and Srivatsav Kunnawalkam Eyavalli for helpful discussions in preparing this paper.

\section{Proofs of theorems}

We first prove Theorem 1.  In fact, the following yields an even stronger result:

\begin{thm}
Suppose that $M$ is an e.c.\ factor with subalgebra $N$.  Then any embedding of $N$ in $M$ is approximately unitarily conjugate to the inclusion map.
\end{thm}

\begin{proof}
Let $f:N\to M$ be an embedding.  Let $P$ be the HNN extension obtained from $M$ and $f$; we note that $P$ is finite and there is a trace on $P$ such that the inclusion $M\subseteq P$ is trace-preserving (see \cite[Corollary 4.2]{ueda}).  In particular, there is a unitary $u\in P$ such that $uf(x)u^*=x$ for all $x\in N$.  Since $M$ is e.c., this implies that for any finite $F\subseteq N$ and $\epsilon>0$, there is a unitary $v\in M$ such that $\|vf(x)v^*-x\|_2<\epsilon$ for all $x\in F$, as desired.
\end{proof}

\begin{remark}[For the model theorists]
Theorem 4 implies that, in any e.c.\ II$_1$ factor, the quantifier-free type of a tuple implies its complete type.  It would be interesting to see if one could leverage this fact to gain any further insight into the class of e.c.\ factors.
\end{remark}

In connection with Theorem 1, we say that a II$_1$ factor $M$ has the \textbf{weak Jung property} if every endomorphism of $M$ is approximately inner.


\begin{prop}
A separable II$_1$ factor has the Jung property if and only if it has the weak Jung property and is self-tracially stable.
\end{prop}

\begin{proof}
First suppose that $M$ has the Jung property and that $f:M\to M$ is an endomorphism.  By viewing $f$ as taking values in $M^\u$, there is a unitary $u\in M^\u$ such that $uf(x)u^*=x$ for all $x\in M$.  In particular, given any finite $F\subseteq M$ and $\epsilon>0$, there is a unitary $v\in M$ such that $\|vf(x)v^*-x\|_2<\epsilon$ for all $x\in F$, whence $f$ is approximately inner.  As mentioned in the introduction, any II$_1$ factor with the Jung property is self-tracially stable. 

The converse is clear.
\end{proof}







Given the fact that $\R$ is the unique separable embeddable factor with either the Jung property or the property of being self-tracially stable, the following question seems natural\footnote{On the other hand, there may be more than one non-embeddable factor with the generalized Jung property; see \cite[Corollary 3.3.5]{agk}.}:

\begin{question}
Must there be at most one non-embeddable factor with the Jung property?  That is self-tracially stable?
\end{question}

We now move on to Theorem 2, which will follow from an alternative characterization of the enforceable factor.  First, recall from \cite{a} that if $\mathfrak C$ is a class of tracial von Neumann algebras, then a tracial von Neumann algebra $M$ is said to be \textbf{$\mathfrak C$-tracially stable} if whenever $f:M\to \prod_\u N_i$ is an embedding with $\u$ a nonprincipal ultrafilter on $\mathbb N$ and each $N_i$ belongs to $\mathfrak C$, then there are *-homomorphisms $f_i:M\to N_i$ such that $f(x)=(f_i(x))_\u$ for all $x\in M$.  (We refer to the sequence $(f_i)_{i\in \mathbb N}$ as an ``approximate lifting'' of $f$.)  In particular, $M$ is self-tracially stable if and only if $M$ is $\{M\}$-tracially stable.

We let $\mathfrak{E}$ denote the class of e.c.\ factors.  The following theorem immediately implies Theorem 2:

\begin{thm}\label{generaltheorem2}
The II$_1$ factor $M$ is the enforceable factor if and only if it is e.c.\ and $\mathfrak{E}$-tracially stable.
\end{thm}

In order to prove Theorem \ref{generaltheorem2}, we need to recall a few model-theoretic facts from \cite{games}.  First, if $M$ is an e.c.\ factor and $a$ is a tuple from $M$, the \textbf{existential type of $a$ in $M$}, denoted $\operatorname{etp}^M(a)$, is the collection of existential formulae $\varphi(x)$ such that $\varphi^M(a)=0$.  Such an existential type is called \textbf{isolated} if, given any $\epsilon>0$, there is an existential formula $\varphi(x)$ and $\delta>0$ such that $\varphi^M(a)=0$ and whenever $N$ is an e.c.\ factor with a tuple $b\in N$ such that $\varphi^N(b)<\delta$, then there is $c\in N$ such that $\|b-c\|_2<\epsilon$ and for which $\operatorname{etp}^M(a)=\operatorname{etp}^N(c)$.  The e.c.\ factor $M$ is called \textbf{e-atomic} if the existential types of all finite tuples are isolated.  It is shown in \cite[Section 6]{games} that an e.c.\ factor $M$ is e-atomic if and only if it is the enforceable factor (in which case it is unique).   

\begin{proof}[Proof of Theorem \ref{generaltheorem2}]
First suppose that $M$ is the enforceable factor.  We must show that $M$ is $\mathfrak{E}$-tracially stable.  Towards this end, fix an embedding $f:M\to \prod_\u N_i$ with each $N_i$ e.c.  Since $M$ is e.c., it is McDuff, whence singly generated.  Fix a generator $a$ of $M$ and write $f(a)=(a_i)_\u$.  Fix $\epsilon>0$ and let $\varphi(x)$ and $\delta>0$ be as in the definition of isolated existential type for $a$ and $\epsilon$.  Since $M$ is e.c., $f$ is an existential embedding, meaning that $\varphi^{\prod_\u N_i}(f(a))=0$ and thus $\varphi^{N_i}(a_i)<\delta$ for $\u$-almost all $i$.  For these $i$, there is $b_i\in N_i$ such that $\|a_i-b_i\|_2<\epsilon$ and for which the map $a\mapsto b_i$ extends to an isomorphism between $M$ and the subalgebra of $N_i$ generated by $b_i$.  Thus, $f$ has an approximate lifting.


Conversely, suppose that $M$ is e.c. and $\mathfrak{E}$-tracially stable.  It follows that $M$ embeds into every e.c. II$_1$ factor, whence $M$ is enforceable.
\end{proof}






Finally, we prove Theorem 3.

\begin{proof}[Proof of Theorem 3]
Let $M$ be the prime model of $T^f$.  To show that $M$ has the Jung property, we show that $M$ has the weak Jung property and is self-tracially stable.

Fix a finitely generic factor $N$.  Since $M$ is the prime model of $T^f$, we have that $M$ embeds elementarily in $N$.  Thus, by \cite[Corollary 3.12]{games}, $M$ itself is finitely generic.  In particular, $M$ is e.c. and thus has the weak Jung property.

It remains to show that $M$ is self-tracially stable.  The argument for showing this is similar to that showing that the enforceable factor is self-tracially stable.  Indeed, fix an embedding $f:M\to M^\u$.  This time, given any $a\in M$, the complete type of $a$ in $M$, denoted $\tp^M(a)$, is isolated\footnote{This follows from the fact that prime models of theories are atomic models.}.  Fix a generator $a$ of $M$ and write $f(a)=(a_i)_\u$.  Given $\epsilon>0$, there is some formula $\varphi(x)$ and $\delta>0$ such that $\varphi^M(a)=0$ and such that, given any model $N$ of $T^f$ and any $b\in N$ with $\varphi^N(b)<\delta$, there is $c\in N$ such that $\operatorname{tp}^M(a)=\operatorname{tp}^N(c)$ and $\|b-c\|_2<\epsilon$.  Since $M$ is finitely generic, $f$ is an elementary map, whence $\varphi^{M^\u}(f(a))=0$ and thus $\varphi^M(a_i)<\delta$ for $\u$-almost all $i$.  As before, for these $i$, this guarantees the existence of $b_i\in M$ such that $\|a_i-b_i\|_2<\epsilon$ and such that the map $a\mapsto b_i$ extends to an embedding of $M$ into itself.  Thus, $f$ has an approximate lift.
\end{proof}

\begin{remark}
It is not clear if there is any relationship between the existence of the enforceable factor and the existence of the prime model of $T^f$.
\end{remark}

\section{Revisiting the embeddable situation}

In this section, we show how our techniques from above can yield different proofs that $\R$ has the Jung property and is self-tracially similar.

Recall from the introduction that $\R$ is the enforceable embeddable factor.  Besides the model theory behind building models by games, the two main operator-algebraic ingredients in the proof are:
\begin{itemize}
    \item Being hyperfinite is $\forall\bigvee\exists$-axiomatizable\footnote{Morally speaking, one just axiomatizes the property that any finite tuple is within any positive tolerance of a copy of some matrix algebra.}, whence being hyperfinite is an enforceable property for embeddable factors.
    \item $\R$ is the unique separable hyperfinite factor.
\end{itemize}

Noting that our proof from the previous section that the enforceable factor (should it exist) is self-tracially stable relativizes immediately to the embeddable situation, we obtain the fact that $\R$ is self-tracially stable, \emph{without resorting to the fact that $\R$ has the Jung property}.

Unfortunately, our proof in the previous section that the enforceable factor (should it exist) has the weak Jung property does not necessarily relativize to the embeddable situation as the following seems to be an open question:

\begin{question}
Is the class of embeddable tracial von Neumann algebras closed under HNN extensions?
\end{question}

If the answer to the previous question is positive, then we learn that all e.c. embeddable factors (and thus, in particular, $\R$ itself) have the weak Jung property.  

Nevertheless, we can give a proof that is similar in spirit that does relativize to the embeddable situation.  Indeed, fix an endomorphism $f:\R\to \R$; we show that $f$ is approximately inner.  Let $a$ be a generator of $\R$.  Since $\R$ is a finitely generic embeddable factor\footnote{This follows from being the enforceable factor.}, we have $\tp^{\R}(a)=\tp^{\R}(f(a))$.  Consequently, there is an elementary extension $N$ of $\R$ and an automorphism $\sigma$ of $N$ such that $\sigma(f(a))=a$.  Using that $\R\subseteq N\rtimes_\sigma \mathbb{Z}$ and the class of embeddable factors is closed under crossed products by $\mathbb Z$ (and, more generally, by any amenable group), we have that, given any $\epsilon>0$, there is a unitary $u\in \R$ such that $\|uf(a)u^*-a\|_2<\epsilon$.  Consequently, $f$ is approximately inner. 

Combining these proofs gives a new proof that $\R$ has the Jung property.

We end with the following natural question which, to the best of our knowledge, is open:

\begin{question}
Is $\R$ the unique embeddable factor with the weak Jung property?
\end{question}

As mentioned above, if the class of embeddable tracial von Neumann algebras is closed under HNN extensions, then there are a plethora of embeddable factors with the weak Jung property.



\begin{thebibliography}{99}
\bibitem{a} S. Atkinson, \emph{Some results on tracial stability and graph products}, to appear in Indiana Univ. Math. J.
\bibitem{ak} S. Atkinson and S. Kunnawalkam Elayavalli, \emph{On ultraproduct embeddings and amenability for tracial von Neumann algebras}, to appear in Int. Math. Res. Not.
\bibitem{agk} S. Atkinson, I. Goldbring, and S. Kunnawalkam Elayavalli, \emph{Factorial relative commutants and the generalized Jung property for II$_1$ factors}, preprint.  arXiv 2004.02293.
\bibitem{connes} A. Connes, \emph{Classification of injective factors. Cases II$_1$, II$_\infty$, III$_\lambda$, $\lambda\not=1$}, Ann. of Math. \textbf{104}(1976), 73–115.
\bibitem{ecfactor} I. Farah, I. Goldbring, B. Hart, and D. Sherman, \emph{Existentially closed II$_1$ factors}, Fundamenta Mathematicae \textbf{233} (2016), 173-196. 
\bibitem{ssa} I. Farah, B. Hart, A. Tikuisis, and M. Rørdam, \emph{Relative commutants of strongly self-absorbing \cstar-algebras}, Selecta Math. \textbf{23} (2017) 363-387.
\bibitem{games} I. Goldbring, \emph{Enforceable operator algebras}, to appear in the Journal of the Institute of Mathematics of Jussieu.
\bibitem{mip}  Z. Ji, A. Natarajan, T. Vidick, J. Wright and H. Yuen, MIP* = RE, preprint, arxiv 2001.04383.
\bibitem{jung} K. Jung, \emph{Amenability, tubularity, and embeddings into $\mathcal{R}^\omega$}, Math. Ann. \textbf{338}(2007), 241–248.
\bibitem{mvn} F.J. Murray and J. von Neumann, \emph{On rings of operators IV}, Ann. of Math. \textbf{44} (1943), 716-808.
\bibitem{ueda} Y. Ueda, \emph{HNN extensions of von Neumann algebras}, J. Funct. Anal. \textbf{225} (2005), 383-426.
\end{thebibliography}
\end{document}